[1994/12/01]
\documentclass{ijmart-mod}
\chardef\bslash=`\\ 

\usepackage{graphicx}
\usepackage[breaklinks=true]{hyperref}
\usepackage{mathtools}

\newtheorem{thm}{Theorem}[section]
\newtheorem{cor}[thm]{Corollary}
\newtheorem{lem}[thm]{Lemma}

\theoremstyle{definition}

\theoremstyle{remark}

\newcommand{\eval}[2][\right]{\relax
  \ifx#1\right\relax \left.\fi#2#1\rvert}

\begin{document}
\title{The diameter of lattice zonotopes}

\author[Antoine Deza]{Antoine Deza}
\address{McMaster University, Hamilton, Ontario, Canada}
\email{deza@mcmaster.ca} 

\author[Lionel Pournin]{Lionel Pournin}
\address{Universit{\'e} Paris 13, Villetaneuse, France}
\email{lionel.pournin@univ-paris13.fr}

\author[Noriyoshi Sukegawa]{Noriyoshi Sukegawa}
\address{Tokyo University of Science, Katsushika-ku, Japan}
\email{sukegawa@rs.tus.ac.jp}

\begin{abstract}
We establish sharp asymptotic estimates for the diameter of primitive zonotopes when their dimension is fixed. We also prove that, for infinitely many integers $k$, the largest possible diameter of a lattice zonotope contained in the hypercube $[0,k]^d$ is uniquely achieved by a primitive zonotope. As a consequence, we obtain that this largest diameter grows like $k^{d/(d+1)}$ up to an explicit multiplicative constant, when $d$ is fixed and $k$ goes to infinity, providing a new lower bound on the largest possible diameter of a lattice polytope contained in $[0,k]^d$.
\end{abstract}
\maketitle

\section{Introduction}\label{sec.DPS.0}

A polytope contained in $\mathbb{R}^d$ is called a lattice polytope when all of its vertices belong to the lattice $\mathbb{Z}^d$. These objects appear in a variety of contexts as, for instance in combinatorial optimization \cite{DelPiaMichini2016,Naddef1989,NillZiegler2011,Sebo1999}, in discrete geometry \cite{AcketaZunic1995,BaranyLarman1998,BaranyProdromou2006,BeckRobins2015,BrionVergne1997,LagariasZiegler1991}, or in combinatorics \cite{BaranyKantor2000,BlancoSantos2018,BlancoSantos2019,Kantor1999,SantosValino2018}. In order to investigate their extremal properties, the lattice polytopes contained in compact convex sets of growing size, such as balls \cite{BaranyLarman1998}, squares \cite{AcketaZunic1995,Thiele1991}, hypercubes \cite{DelPiaMichini2016,DezaManoussakisOnn2018}, or arbitrary $2$-dimensional compact convex sets \cite{BaranyProdromou2006} are often considered. For instance, the largest possible number of vertices $\phi(2,k)$ of a lattice polygon contained in the square $[0,k]^2$ is known to behave as
\begin{equation}\label{DPS.eq.1}
\phi(2,k)\sim\frac{12}{(2\pi)^{2/3}}k^{2/3}
\end{equation}
when $k$ goes to infinity \cite{AcketaZunic1995,Thiele1991}. In higher dimension, a similar result can be obtained from \cite{Andrews1963} and from \cite{BaranyLarman1998}. More precisely, the largest possible number of vertices $\phi(d,k)$ of a lattice polytope contained in the hypercube $[0,k]^d$ grows like $k^{d(d-1)/(d+1)}$ up to a multiplicative term that only depends on $d$. Note that no expression is known for this multiplicative term when $d\geq3$.

Another quantity that has attracted attention, due to its connection with the complexity of the simplex algorithm \cite{KalaiKleitman1992,Santos2012,Sukegawa2019,Sukegawa2017}, is the largest diameter $\delta(d,k)$ a lattice polytope contained in $[0,k]^d$ can have \cite{DelPiaMichini2016,DezaManoussakisOnn2018,DezaPournin2018,KleinschmidtOnn1992,Naddef1989}. Here, by the diameter of a polytope, we mean the diameter of the graph made of its vertices and edges. Since a polygon with $n$ vertices has diameter $\lfloor{n/2}\rfloor$, estimating the asymptotic behavior of $\phi(d,k)$ and $\delta(d,k)$ when $k$ goes to infinity can be considered two generalizations of (\ref{DPS.eq.1}) to higher dimensions. It is shown in \cite{Naddef1989} that $\delta(d,1)=d$, in \cite{DelPiaMichini2016} that $\delta(d,2)=\lfloor3d/2\rfloor$, in \cite{DezaPournin2018} that
$$
\delta(d,k)\leq{kd-\left\lceil\frac{2d}{3}\right\rceil-(k-3)}\mbox{ when }k\geq3\mbox{,}
$$
and in \cite{DezaManoussakisOnn2018} that
\begin{equation}\label{DPS.eq.2}
\delta(d,k)\geq\left\lfloor\frac{(k+1)d}{2}\right\rfloor\mbox{ when }k<2d\mbox{.}
\end{equation}

This lower bound on $\delta(d,k)$ is obtained using a particular family of lattice zonotopes, referred to as \emph{primitive zonotopes}. Recall that a zonotope is the Minkowski sum of pairwise non-collinear line segments, which we call its generators. Informally, primitive zonotopes are generated by the shortest possible lattice segments. In particular these segments themselves are primitive in the sense that the only lattice points they contain are their extremities. A formal definition of primitive zonotopes will be given in Section \ref{sec.DPS.1}. 
In this paper, we provide the asymptotic diameter of the primitive zonotopes defined in \cite{DezaManoussakisOnn2018} when their dimension is fixed while the number of their generators goes to infinity. We also show that, for infinitely-many integers $k$, the largest possible diameter $\delta_z(d,k)$ of a lattice zonotope contained in the hypercube $[0,k]^d$ is uniquely achieved by a primitive zonotope. As a first consequence, we partially answer the following question, posed by G{\"u}nter Rote: \emph{how can one compute $\delta_z(d,k)$}? In addition, we establish the following asymptotic estimate for $\delta_z(d,k)$. In the statement of this result, $\zeta$ stands for Riemann's zeta function.
\begin{thm}\label{DPS.thm.main.1}
For any fixed $d$, the largest possible diameter of a lattice zonotope contained in the hypercube $[0,k]^d$ satisfies
$$
\displaystyle\delta_z(d,k)\sim c(d)k^\frac{d}{d+1}\mbox{,}
$$
when $k$ goes to infinity, where $\displaystyle{c(d)=\left(\frac{2^d(d+1)^d}{2\,d!\zeta(d)}\right)^\frac{1}{d+1}}$.
\end{thm}

Theorem \ref{DPS.thm.main.1} can be thought of as a generalization of (\ref{DPS.eq.1}) to zonotopes of arbitrary dimension because, as mentioned above, the number of vertices of a polygon is roughly twice its diameter. In particular, $c(2)$ is half the multiplicative constant in (\ref{DPS.eq.1}). Theorem \ref{DPS.thm.main.1} also immediately provides a lower bound on $\delta(d,k)$ similar to (\ref{DPS.eq.2}), except that it is valid when $k$ goes to infinity.

\begin{cor}\label{DPS.cor.main.1}
For any fixed $d$, $\displaystyle\delta(d,k)\geq c(d)k^\frac{d}{d+1}+o(1)$.
\end{cor}

It is conjectured in \cite{DezaManoussakisOnn2018} that $\delta(d,k)$ is achieved for all $d$ and $k$ by a lattice zonotope generated by primitive segments. Hence, Corollary \ref{DPS.cor.main.1} conjecturally provides the correct asymptotic estimate for $\delta(d,k)$.

The asymptotic diameter of primitive zonotopes will be established in Section~\ref{sec.DPS.1}. The proof that the unique lattice zonotope of diameter $\delta_z(d,k)$ contained in the hypercube $[0,k]^d$ is a primitive zonotope for infinitely-many values of $k$ is given in Section \ref{sec.DPS.2}. Finally, Theorem \ref{DPS.thm.main.1} is proven in Section \ref{sec.DPS.3}.

\section{The asymptotic diameter of primitive zonotopes}\label{sec.DPS.1}

We first recall the formal definition of primitive zonotopes \cite{DezaManoussakisOnn2018}. Call a point in the lattice $\mathbb{Z}^d$ \emph{primitive} when it is not equal to $0$ and the greatest common divisor of its coordinates is equal to $1$. In other words, the line segment that connects the origin of $\mathbb{R}^d$ to such a point is primitive in the sense given in the previous section. The set of the primitive points contained in $\mathbb{Z}^d$ will be denoted by $\mathbb{P}^d$ in the sequel. We will also refer to the $d$-dimensional ball of radius $p$ centered in $0$ for the $q$-norm as $B_q(d,p)$.

A first family of primitive zonotopes, denoted by $H_q(d,p)$, is defined in \cite{DezaManoussakisOnn2018} as the lattice zonotopes whose generators are the segments incident to $0$ on one end and to a point in $\mathbb{P}^d\cap{B_q(d,p)}$ whose first non-zero coordinate is positive on the other. Another family, referred to as $H_q^+(d,p)$ are the lattice zonotopes whose generators are the generators of $H_q(d,p)$ contained in the positive orthant $[0,+\infty[^d$. A useful property of zonotopes is that their diameter is equal to the number of their generators \cite{Ziegler1995}. Therefore, in order to determine the diameter of $H_q(d,p)$ and $H_q^+(d,p)$, we only need to count the primitive points in $B_q(d,p)$ at the extremity of their generators. 
For instance, $H_1(d,2)$ has diameter $d^2$. In this section we provide the asymptotic diameter of both $H_q(d,p)$ and $H_q^+(d,p)$ for any fixed $d$ and $q$, when $p$ goes to infinity; that is, when the radius of the ball the generators of these zonotopes are picked from grows large.

It is well known that the density of the primitive points in the lattice is $1/\zeta(d)$ \cite{HardyWright1938,KranakisPocchiola1994,Nymann1970}. The following result is proven in \cite{KranakisPocchiola1994} (see also the remark in Page~4 of \cite{BaranyMartinNaslundRobins2015}). In the statement of this result, $C$ is any convex compact subset of $\mathbb{R}^d$ that contains the origin and whose interior is non-empty and $\mathrm{vol}(pC)$ stands for the volume of the dilation of $C$ by a coefficient $p$.

\begin{lem}\label{sec.DPS.1.lem.1}
$\displaystyle\lim_{p\to\infty}\frac{\left|pC\cap\mathbb{P}^d\right|}{\mathrm{vol}(pC)}=\frac{1}{\zeta(d)}$.
\end{lem}

It is also well known that
\begin{equation}\label{sec.DPS.1.eq.1}
\mathrm{vol}(B_q(d,p))=\frac{\left(2\Gamma\!\left(\frac{1}{q}+1\right)\!p\right)^d}{\Gamma\!\left(\frac{d}{q}+1\right)}\mbox{,}
\end{equation}
where $\Gamma$ denotes Euler's gamma function.

In the remainder of the article, we refer to the diameter of a polytope $P$ as $\delta(P)$. The following is obtained by combining (\ref{sec.DPS.1.eq.1}) with Lemma \ref{sec.DPS.1.lem.1}.

\begin{thm}\label{sec.DPS.1.thm.1}
$\displaystyle\lim_{p\to\infty}\frac{\delta\!\left(H_q(d,p)\right)}{p^d}=\frac{\left(2\Gamma\!\left(\frac{1}{q}+1\right)\right)^d}{2\Gamma\!\left(\frac{d}{q}+1\right)\!\zeta(d)}$.
\end{thm}
\begin{proof}
Recall that the diameter of $H_q(d,p)$ is equal to the number of its generators; that is, to the number of the primitive points in $B_q(d,p)$ whose first non-zero coordinate is positive or, equivalently, to half the number of the primitive points contained in $B_q(d,p)$. As a consequence,
\begin{equation}\label{sec.DPS.1.thm.1.eq.1}
\frac{\delta(H_q(d,p))}{p^d}=\frac{\!\left|B_q(d,p)\cap\mathbb{P}^d\right|}{2p^d}\mbox{.}
\end{equation}

Taking $pC=B_q(d,p)$ in the statement of Lemma \ref{sec.DPS.1.lem.1} and evaluating the volume of this ball using equation (\ref{sec.DPS.1.eq.1}) yields
\begin{equation}\label{sec.DPS.1.thm.1.eq.2}
\lim_{p\to\infty}\frac{\left|B_q(d,p)\cap\mathbb{P}^d\right|}{p^d}=\frac{\left(2\Gamma\!\left(\frac{1}{q}+1\right)\right)^d}{\Gamma\!\left(\frac{d}{q}+1\right)\!\zeta(d)}\mbox{.}
\end{equation}

Combining equalities (\ref{sec.DPS.1.thm.1.eq.1}) and (\ref{sec.DPS.1.thm.1.eq.2}) completes the proof.
\end{proof}

We now turn our attention to the primitive zonotopes $H_q^+(d,p)$.

\begin{thm}\label{sec.DPS.1.thm.2}
$\displaystyle\lim_{p\to\infty}\frac{\delta\!\left(H_q^+(d,p)\right)}{p^d}=\frac{\Gamma\!\left(\frac{1}{q}+1\right)^d}{\Gamma\!\left(\frac{d}{q}+1\right)\!\zeta(d)}$.
\end{thm}
\begin{proof}
Denote by $a_q(d,p)$ the number of generators of $H_q^+(d,p)$ that are not contained in any face of dimension less than $d$ of the cone $[0,+\infty[^d$. 
Consider a face $F$ of $[0,+\infty[^d$. Assume that $F$ is $i$-dimensional with $i\geq0$ and observe that there are exactly $a_q(i,p)$ generators of $H_q^+(d,p)$ contained in $F$ but not in any face of $[0,+\infty[^d$ of dimension less than $i$. Further observe that $[0,+\infty[^d$ has exactly $d\choose{i}$ faces of dimension $i$. Since the diameter of $H_q^+(d,p)$ is the number of its generators; that is $\left|B_q(d,p)\cap[0,+\infty[^d\right|$, we therefore obtain
\begin{equation}\label{sec.DPS.1.thm.2.eq.1}
\delta\!\left(H_q^+(d,p)\right)=\sum_{i=1}^d{d\choose i}a_q(i,p)\mbox{.}
\end{equation}

Now consider the $2^d$ orthants of $\mathbb{R}^d$. These orthants are polyhedral cones, and the faces of these cones collectively form a polyhedral subdivision of $\mathbb{R}^d$. The number of $i$-dimensional polyhedra in this subdivision is equal to the number of the $(i-1)$-dimensional faces of a $d$-dimensional cross-polytope; that is, $2^i{d\choose i}$. Hence, the number of primitive lattice points in $B_q(d,p)$ satisfies
\begin{equation}\label{sec.DPS.1.thm.2.eq.0}
\left|B_q(d,p)\cap\mathbb{P}^d\right|=\sum_{i=1}^d2^i{d\choose i}a_q(i,p)\mbox{.}
\end{equation}

The equalities (\ref{sec.DPS.1.thm.2.eq.1}) and (\ref{sec.DPS.1.thm.2.eq.0}) yield
\begin{equation}\label{sec.DPS.1.thm.2.eq.2}
2^d\delta\!\left(H_q^+(d,p)\right)=\left|B_q(d,p)\cap\mathbb{P}^d\right|+\sum_{i=1}^{d-1}(2^d-2^i){d\choose i}a_q(i,p)\mbox{.}
\end{equation}

Now observe that
$$
0\leq\sum_{i=1}^{d-1}(2^d-2^i){d\choose i}a_q(i,p)\leq{d2^{d-1}\sum_{i=1}^{d-1}{d-1\choose i}a_q(i,p)}\mbox{.}
$$

According to (\ref{sec.DPS.1.thm.2.eq.1}), the right-hand side of the second inequality can be expressed in terms of $\delta\!\left(H_q^+(d-1,p)\right)$. Therefore, by (\ref{sec.DPS.1.thm.2.eq.2}),
\begin{equation}\label{sec.DPS.1.thm.2.eq.3}
0\leq2^d\delta\!\left(H_q^+(d,p)\right)-\left|B_q(d,p)\cap\mathbb{P}^d\right|\leq{d2^{d-1}\delta\!\left(H_q^+(d-1,p)\right)}\mbox{.}
\end{equation}

This double inequality makes it possible to prove the theorem by induction on $d$. Indeed, under the inductive property that
$$
\displaystyle\lim_{p\to\infty}\frac{\delta\!\left(H_q^+(d-1,p)\right)}{p^d}=0\mbox{,}
$$
the combination of (\ref{sec.DPS.1.eq.1}), (\ref{sec.DPS.1.thm.2.eq.3}), and Lemma \ref{sec.DPS.1.lem.1} provides the desired result. We still need to establish the base case for the induction. Observe that $a_q(1,p)=1$. As a consequence, equalities (\ref{sec.DPS.1.thm.2.eq.1}) and (\ref{sec.DPS.1.thm.2.eq.0}) yield
$$
\left|B_q(2,p)\cap\mathbb{P}^2\right|=4\delta\!\left(H_q^+(2,p)\right)-4\mbox{.}
$$

Combining this with (\ref{sec.DPS.1.eq.1}) and Lemma \ref{sec.DPS.1.lem.1} proves the result when $d=2$.
\end{proof}

\section{Lattice zonotopes with the largest possible diameter}\label{sec.DPS.2}

For any $d$-dimensional lattice polytope $P$, we denote by $k(P)$ the smallest integer $k$ such that some translate of $P$ by a lattice vector is contained in the hypercube $[0,k]^d$. We show in this section that, for all $p$, the unique lattice polytope with diameter $\delta_z(d,k(H_1(d,p)))$ contained in the hypercube $[0,k(H_1(d,p))]^d$ is a translate of $H_1(d,p)$. Note that the primitive zonotopes $H_1(d,p)$ and the $1$-norm they are built from play an important role here.

\begin{thm}\label{sec.DPS.2.lem.2}
Consider a $d$-dimensional lattice zonotope $Z$ and a positive integer $p$. If $\delta\!\left(H_1(d,p)\right)\leq\delta(Z)$, then $k\!\left(H_1(d,p)\right)\leq{k(Z)}$. If in addition, the first of these inequalities is strict then so is the second one, and if both inequalities are equalities, then $Z$ is a translate of $H_1(d,p)$. 
\end{thm}
\begin{proof}
Denote by $\mathcal{Z}$ the set of the generators of $Z$ and by $\mathcal{G}$ the set of the generators of $H_1(d,p)$. We will assume that the generators of $Z$ are all incident to the origin of $\mathbb{R}^d$ and that the first non-zero coordinate of their other vertex is positive. This can be done without loss of generality by translating the generators of $Z$ or equivalently, by translating $Z$ itself. Recall that the diameter of a zonotope is equal to the number of its generators. Therefore, if the diameter of $H_1(d,p)$ is not greater than that of $Z$, then there exists an injection $\psi$ from $\mathcal{G}$ into $\mathcal{Z}$. Since distinct generators of a zonotope are never collinear, we can require this injection to be such that, if a generator $z$ of $Z$ is collinear to a generator $g$ of $H_1(d,p)$, then $\psi(g)$ is equal to $z$. It follows from this assumption that the $1$-norm of a generator $g$ of $H_1(d,p)$ is never greater than the $1$-norm of $\psi(g)$. Moreover, by construction, if these $1$-norms are equal, then $g$ and $\psi(g)$ must coincide. Now observe that the $1$-norms of the generators of $Z$ sum to at most $k(Z)d$ and, since $H_1(d,p)$ is invariant up to translation by the isometries of $\mathbb{R}^d$ that consist in permuting coordinates, the $1$-norms of the generators of $H_1(d,p)$ sum to exactly $k(H_1(d,p))d$. As a consequence,
$$
k\!\left(H_1(d,p)\right)\!d\leq{k(Z)}d\mbox{,}
$$
and this inequality is strict when the number of generators of $Z$ is greater than that of $H_1(d,p)$. Dividing this inequality by $d$ provides the first part of the theorem. Now observe that, if $\delta\!\left(H_1(d,p)\right)=\delta(Z)$, then $\psi$ must be a bijection. If in addition, $k\!\left(H_1(d,p)\right)=k(Z)$, the $1$-norm of a generator of $H_1(d,p)$ is necessarily equal to the $1$-norm of its image by $\psi$. In this case, $Z$ and $H_1(d,p)$ have exactly the same generators and they must coincide.
\end{proof}

Theorem \ref{sec.DPS.2.lem.2} will be one of the main ingredients in the proof of Theorem~\ref{DPS.thm.main.1}. It also admits the following consequence, announced above.

\begin{cor}\label{sec.DPS.2.cor.0}
The unique lattice zonotope of diameter $\delta_z(d,k(H_1(d,p)))$ contained in the hypercube $[0,k(H_1(d,p))]^d$ is a translate of $H_1(d,p)$.
\end{cor}
\begin{proof}
Consider a lattice zonotope $Z$ with diameter $\delta_z(d,k(H_1(d,p)))$ contained in the hypercube $[0,k(H_1(d,p))]^d$. In particular,
\begin{equation}\label{sec.DPS.2.cor.0.eq.1}
k(Z)\leq{k(H_1(d,p))}\mbox{.}
\end{equation}

By definition, $H_1(d,p)$ is a lattice zonotope contained, up to translation in $[0,k(H_1(d,p))]^d$. Therefore, its diameter must be at most $\delta_z(d,k(H_1(d,p)))$. In other words, $\delta(H_1(d,p))\leq\delta(Z)$. Hence, by Theorem \ref{sec.DPS.2.lem.2}, 
\begin{equation}\label{sec.DPS.2.cor.0.eq.2}
k(H_1(d,p))\leq{k(Z)}\mbox{.}
\end{equation}

According to (\ref{sec.DPS.2.cor.0.eq.1}) and (\ref{sec.DPS.2.cor.0.eq.2}), $k(H_1(d,p))$ and $k(Z)$ coincide. Since the diameter of $H_1(d,p)$ is not greater than that of $Z$, it therefore follows from Theorem \ref{sec.DPS.2.lem.2} that these diameters also coincide. Invoking Theorem \ref{sec.DPS.2.lem.2} a third time, we obtain that $Z$ is a translate of $H_1(d,p)$. It remains to show that there is only one translate of $H_1(d,p)$ contained in the hypercube $[0,k(H_1(d,p))]^d$. This is an immediate consequence of $H_1(d,p)$ being invariant up to translation by the isometries of $\mathbb{R}^d$ that consist in permuting coordinates.
\end{proof}

Corollary \ref{sec.DPS.2.cor.0} provides a way to determine $\delta_z(d,k)$ when, for some integer $p$, $k$ coincides with $k\!\left(H_1(d,p)\right)$. This partially answers a question posed by G{\"u}nter Rote. Indeed, in this case, it follows from Corollary~\ref{sec.DPS.2.cor.0} that it suffices count the points in $\mathbb{P}^d\cap{B_1(d,p)}$ whose first non-zero coordinate is positive. 

\section{An upper bound on the diameter of lattice zonotopes}\label{sec.DPS.3}

This section is devoted to proving Theorem \ref{DPS.thm.main.1}. In order to do that, we first relate $k\!\left(H_1(d,p)\right)$ and $\delta\!\left(H_1(d,p)\right)$ as follows.

\begin{lem}\label{sec.DPS.2.lem.1}
$\displaystyle{k\!\left(H_1(d,p)\right)\!d=p\delta\!\left(H_1(d,p)\right)-\sum_{i=0}^{p-1}\delta\!\left(H_1(d,i)\right)}$.
\end{lem}
\begin{proof}
Recall that the diameter of $H_1(d,p)$ is the number of primitive lattice points in $B_1(d,p)$ whose first non-zero coordinate is positive. Hence, when $1\leq{i}\leq{p}$, the number of these points whose $1$-norm is equal to $i$ is
$$
\delta\!\left(H_1(d,i)\right)-\delta\!\left(H_1(d,i-1)\right)\mbox{.}
$$

By symmetry, the $1$-norms of the primitive lattice points in $B_1(d,p)$ whose first non-zero coordinate is positive sum to $k(H_1(d,p))d$. As a consequence,
$$
k(H_1(d,p))d=\sum_{i=1}^pi\!\left[\delta\!\left(H_1(d,i)\right)-\delta\!\left(H_1(d,i-1)\right)\right]\mbox{.}
$$

Rearranging the right-hand side of this equality completes the proof.
\end{proof}

The following theorem provides the asymptotic behavior of $k\!\left(H_1(d,p)\right)$.

\begin{thm}\label{thm.DPS.1}
$\displaystyle\lim_{p\to\infty}\frac{k\!\left(H_1(d,p)\right)}{p^{d+1}}=\frac{2^{d-1}}{(d+1)!\zeta(d)}$.
\end{thm}
\begin{proof}
First observe that, according to Theorem \ref{sec.DPS.1.thm.1},
\begin{equation}\label{thm.DPS.1.eq.1}
\left|\frac{\delta(H_1(d,p))}{p^d}-\frac{2^{d-1}}{d!\zeta(d)}\right|\leq\varepsilon(p)\mbox{,}
\end{equation}
where $\varepsilon:\mathbb{N}\rightarrow\mathbb{R}$ is a function such that
$$
\lim_{p\to\infty}\varepsilon(p)=0\mbox{.}
$$

We can further assume without loss of generality that $\varepsilon$ is decreasing. Invoking Lemma \ref{sec.DPS.2.lem.1}, and using (\ref{thm.DPS.1.eq.1}), one obtains
\begin{equation}\label{thm.DPS.1.eq.3}
\left|\frac{k\!\left(H_1(d,p)\right)\!d}{p^{d+1}}-\frac{2^{d-1}}{d!\zeta(d)}\left(1-\frac{1}{p^{d+1}}\sum_{i=1}^{p-1}i^d\right)\right|\leq{\varepsilon(p)+\frac{1}{p^{d+1}}\sum_{i=1}^{p-1}\varepsilon(i)i^d}
\end{equation}

However, by Faulhaber's formula,
\begin{equation}\label{lem.DPS.0.eq.3}
\sum_{i=1}^{p-1}i^d=\frac{1}{d+1}p^{d+1}+N(p)\mbox{,}
\end{equation}
where $N(p)$ is a polynomial of degree at most $d$ in $p$. Therefore,
$$
\lim_{p\to\infty}\frac{2^{d-1}}{d!\zeta(d)}\!\left(1-\frac{1}{p^{d+1}}\sum_{i=1}^{p-1}i^d\right)=\frac{d2^{d-1}}{(d+1)!\zeta(d)}\mbox{.}
$$

Now observe that
$$
\lim_{p\to\infty}\frac{1}{p^{d+1}}\sum_{i=1}^{p-1}\varepsilon(i)i^d=0\mbox{.}
$$

Indeed,
$$
\frac{1}{p^{d+1}}\sum_{i=1}^{\lfloor\sqrt{p}\rfloor}\varepsilon(i)i^d\leq\frac{\varepsilon(1)}{p^{(d+1)/2}}\mbox{,}
$$
and
$$
\frac{1}{p^{d+1}}\sum_{i={\lceil\!\sqrt{p}\rceil}}^{p-1}\varepsilon(i)i^d\leq\varepsilon(\lceil\!\sqrt{p}\rceil)\mbox{.}
$$

Hence, letting $p$ go to infinity in (\ref{thm.DPS.1.eq.3}) provides the desired limit.
\end{proof}

The following result is a consequence of Theorems \ref{sec.DPS.1.thm.1} and \ref{thm.DPS.1}. It provides the exact asymptotic behavior of the diameter of $H_1(d,p)$ in terms of $k\!\left(H_1(d,p)\right)$ when $d$ is fixed and $p$ goes to infinity.

\begin{cor}\label{cor.DPS.1}
$\displaystyle\lim_{p\to\infty}\frac{\delta\!\left(H_1(d,p)\right)^{d+1}}{k\!\left(H_1(d,p)\right)^d}=\frac{2^d(d+1)^d}{2\,d!\zeta(d)}$.
\end{cor}
\begin{proof}
By Theorem \ref{sec.DPS.1.thm.1},
$$
\lim_{p\to\infty}\frac{\delta\!\left(H_1(d,p)\right)^{d+1}}{p^{d(d+1)}}=\!\left(\frac{2^{d-1}}{d!\zeta(d)}\right)^{d+1}\mbox{,}
$$
and by Theorem \ref{thm.DPS.1},
$$
\lim_{p\to\infty}\frac{k\!\left(H_1(d,p)\right)^d}{p^{d(d+1)}}=\!\left(\frac{2^{d-1}}{(d+1)!\zeta(d)}\right)^d\mbox{.}
$$

Combining these equalities provides the desired result.
\end{proof}

Note that, for any fixed $d$, $\delta(d,k)$ is an increasing function of $k$. Therefore, Corollary \ref{DPS.cor.main.1} is immediately obtained from Corollary \ref{cor.DPS.1}.

We are finally ready to prove Theorem \ref{DPS.thm.main.1}.

\begin{proof}[Proof of Theorem \ref{DPS.thm.main.1}]
Consider a lattice zonotope $Z$ contained in the hypercube $[0,k]^d$, where $k$ is positive. We can assume without loss of generality that the diameter of $Z$ is not less than the diameter of $H_1(d,1)$. Since
$$
\lim_{p\to\infty}\delta\!\left(H_1(d,p)\right)=+\infty\mbox{,}
$$
there exists a non-negative integer $p$ such that
$$
\delta\!\left(H_1(d,p)\right)\leq\delta(Z)\leq\delta\!\left(H_1(d,p+1)\right)\mbox{.}
$$

According to the first inequality and to Theorem \ref{sec.DPS.2.lem.2},
$$
k\!\left(H_1(d,p)\right)\leq{k}\mbox{.}
$$

Therefore, it follows from the second inequality that
\begin{equation}\label{thm.DPS.2.eq.2}
\frac{\delta(Z)}{k^\frac{d}{d+1}}\leq\frac{\delta\!\left(H_1(d,p+1)\right)}{k\!\left(H_1(d,p)\right)^\frac{d}{d+1}}\mbox{.}
\end{equation}

By Corollary \ref{cor.DPS.1}, 
$$
\lim_{p\to\infty}\frac{\delta\!\left(H_1(d,p)\right)}{k\!\left(H_1(d,p)\right)^\frac{d}{d+1}}=\!\left(\frac{2^d(d+1)^d}{2\,d!\zeta(d)}\right)^\frac{1}{d+1}\mbox{,}
$$
and by Theorem \ref{sec.DPS.1.thm.1},
$$
\lim_{p\to\infty}\frac{\delta\!\left(H_1(d,p+1)\right)}{\delta\!\left(H_1(d,p)\right)}=1\mbox{.}
$$

Therefore, the right-hand side of (\ref{thm.DPS.2.eq.2}) satisfies
\begin{equation}\label{thm.DPS.2.eq.3}
\lim_{p\to\infty}\frac{\delta\!\left(H_1(d,p+1)\right)}{k\!\left(H_1(d,p)\right)^\frac{d}{d+1}}=\!\left(\frac{2^d(d+1)^d}{2\,d!\zeta(d)}\right)^\frac{1}{d+1}\mbox{.}
\end{equation}

By construction, when $k$ goes to infinity, so does $p$. Hence, combining inequality (\ref{thm.DPS.2.eq.2}) with equation (\ref{thm.DPS.2.eq.3}) provides the desired result.
\end{proof}

\noindent{\bf Acknowledgements.} We thank G{\"u}nter Rote, whose question concerning the computation of the diameter of lattice zonotopes was a starting point for this paper, and Imre B\'{a}r\'{a}ny for pointing out reference \cite{KranakisPocchiola1994}. Antoine Deza is partially supported by the Natural Sciences and Engineering Research Council of Canada Discovery Grant program (RGPIN-2015-06163). Lionel Pournin is partially supported by the ANR project SoS (Structures on Surfaces), grant number ANR-17-CE40-0033 and by the PHC project number 42703TD. Noriyoshi Sukegawa is partially supported by the Japan Society for the Promotion of Science (JSPS) Grant-in-Aid for Science Research (A) 26242027.

\bibliography{LatticeZonotopes}

\providecommand{\MR}{\relax\ifhmode\unskip\space\fi MR }
\providecommand{\MRhref}[2]{%
  \href{http://www.ams.org/mathscinet-getitem?mr=#1}{#2}
}
\providecommand{\href}[2]{#2}
\begin{thebibliography}{10}

\bibitem{AcketaZunic1995}
Dragan Acketa and Jovi\v{s}a \v{Z}uni\'{c}, \textsl{On the maximal number of
  edges of convex digital polygons included into an $m\times{m}$-grid}, Journal
  of Combinatorial Theory A \textbf{69} (1995), 358--368.

\bibitem{Andrews1963}
George~E. Andrews, \textsl{A lower bound for the volume of strictly convex
  bodies with many boundary lattice points}, Transactions of the American
  Mathematical Society \textbf{106} (1963), 270--279.

\bibitem{BaranyKantor2000}
Imre B{\'a}r{\'a}ny and Jean-Michel Kantor, \textsl{On the number of lattice
  free polytopes}, European Journal of Combinatorics \textbf{21} (2000), no.~1,
  103--110.

\bibitem{BaranyLarman1998}
Imre B{\'a}r{\'a}ny and David~G. Larman, \textsl{The convex hull of the integer
  points in a large ball}, Mathematische Annalen \textbf{312} (1998), no.~1,
  167--181.

\bibitem{BaranyMartinNaslundRobins2015}
Imre B{\'a}r{\'a}ny, Greg Martin, Eric Naslund, and Sinai Robins,
  \textsl{Primitive points in lattice polygons}, arXiv:1509.02201 (2015).

\bibitem{BaranyProdromou2006}
Imre B{\'a}r{\'a}ny and Maria Prodromou, \textsl{On maximal convex lattice
  polygons inscribed in a plane convex set}, Israel Journal of Mathematics
  \textbf{154} (2006), 337--360.

\bibitem{BeckRobins2015}
Matthias Beck and Sinai Robins, \textsl{Computing the continuous discretely},
  Undergraduate Texts in Mathematics, Springer, 2015.

\bibitem{BlancoSantos2018}
M{\'o}nica Blanco and Francisco Santos, \textsl{Enumeration of lattice
  $3$-polytopes by their number of lattice points}, Discrete \& Computational
  Geometry \textbf{60} (2018), 756--800.

\bibitem{BlancoSantos2019}
M{\'o}nica Blanco and Francisco Santos, \textsl{Non-spanning lattice
  $3$-polytopes}, Journal of Combinatorial Theory A \textbf{161} (2019),
  112--133.

\bibitem{BrionVergne1997}
Michel Brion and Mich{\`e}le Vergne, \textsl{Lattice points in simple
  polytopes}, Journal of the American Mathematical Society \textbf{10} (1997),
  no.~2, 371--392.

\bibitem{DelPiaMichini2016}
Alberto {D}el {P}ia and Carla Michini, \textsl{On the diameter of lattice
  polytopes}, Discrete \& Computational Geometry \textbf{55} (2016), 681--687.

\bibitem{DezaManoussakisOnn2018}
Antoine Deza, George Manoussakis and Shmuel Onn, \textsl{Primitive zonotopes},
  Discrete \& Computational Geometry \textbf{60} (2018), 27--39.

\bibitem{DezaPournin2018}
Antoine Deza and Lionel Pournin, \textsl{Improved bounds on the diameter of
  lattice polytopes}, Acta Mathematica Hungarica \textbf{154} (2018), 457--469.

\bibitem{HardyWright1938}
Godfrey~H. Hardy and Edward~M. Wright, \textsl{An introduction to the theory of
  numbers}, Clarendon Press, Oxford, 1938.

\bibitem{KalaiKleitman1992}
Gil Kalai and Daniel Kleitman, \textsl{A quasi-polynomial bound for the
  diameter of graphs of polyhedra}, Bulletin of the American Mathematical
  Society \textbf{26} (1992), 315--316.

\bibitem{Kantor1999}
Jean-Michel Kantor, \textsl{On the width of lattice-free simplices}, Compositio
  Mathematica \textbf{118} (1999), no.~3, 235--241.

\bibitem{KleinschmidtOnn1992}
Peter Kleinschmidt and Shmuel Onn, \textsl{On the diameter of convex
  polytopes}, Discrete Mathematics \textbf{102} (1992), 75--77.

\bibitem{KranakisPocchiola1994}
Evangelos Kranakis and Michel Pocchiola, \textsl{Counting problems relating to
  a theorem of {Dirichlet}}, Computational Geometry \textbf{4} (1994),
  309--325.

\bibitem{LagariasZiegler1991}
Jeffrey~C. Lagarias and G{\"u}nter~M. Ziegler, \textsl{Bounds for lattice
  polytopes containing a fixed number of interior points in a sublattice},
  Canadian Journal of Mathematics \textbf{43} (1991), 1022--1035.

\bibitem{Naddef1989}
Dennis Naddef, \textsl{The {H}irsch conjecture is true for $(0,1)$-polytopes},
  Mathematical Programming \textbf{45} (1989), 109--110.

\bibitem{NillZiegler2011}
Benjamin Nill and G{\"u}nter~M. Ziegler, \textsl{Projecting lattice polytopes
  without interior lattice points}, Mathematics of Operations Research
  \textbf{36} (2011), 462--467.

\bibitem{Nymann1970}
James~E. Nymann, \textsl{On the probability that $k$ positive integers are
  relatively prime}, Journal of Number Theory \textbf{4} (1972), 469--473.

\bibitem{Santos2012}
Francisco Santos, \textsl{A counterexample to the {Hirsch} conjecture}, Annals
  of Mathematics \textbf{176} (2012), 383--412.

\bibitem{SantosValino2018}
Francisco Santos and {\'O}scar~Iglesias Vali{\~n}o, \textsl{Classification of
  empty lattice 4-simplices of width larger than two}, Transactions of the
  American Mathematical Society \textbf{371} (2019), 6605--6625.

\bibitem{Sebo1999}
Andr{\'a}s Seb\H{o}, \textsl{An introduction to empty simplices}, Proceedings
  of IPCO 7, Lecture Notes in Computer Science, vol. 1610, 1999, pp.~400--414.

\bibitem{Sukegawa2019}
Noriyoshi Sukegawa, \textsl{An asymptotically improved upper bound on the
  diameter of polyhedra}, Discrete \& Computational Geometry, \emph{to appear}.

\bibitem{Sukegawa2017}
Noriyoshi Sukegawa, \textsl{Improving bounds on the diameter of a polyhedron in
  high dimensions}, Discrete Mathematics \textbf{340} (2017), 2134--2142.

\bibitem{Thiele1991}
Torsten Thiele, \textsl{Extremalprobleme f\"{u}r {P}unktmengen}, Diplomarbeit,
  Freie Universit\"{a}t Berlin (1991).

\bibitem{Ziegler1995}
G{\"u}nter~M. Ziegler, \textsl{Lectures on polytopes}, Graduate Texts in
  Mathematics, vol. 152, Springer, 1995.

\end{thebibliography}
\bibliographystyle{ijmart}

\end{document}